\documentclass[10pt, reqno]{amsart} 
\usepackage[T2A]{fontenc}
\usepackage[utf8]{inputenc}
\usepackage[english]{babel}
\usepackage[all]{xy} 
\usepackage{amsmath,amsthm, amssymb,fancyhdr,mathrsfs,color,epsfig, color, graphicx}
\usepackage[usenames,dvipsnames,svgnames,table]{xcolor}
\usepackage[colorlinks=true,hyperindex, pagebackref=true, citecolor=NavyBlue,linkcolor=magenta]{hyperref}

\newtheorem{theorem}[subsection]{Theorem}
\newtheorem{lemma}[subsection]{Lemma}

\newtheorem{remark}[subsection]{Remark}

\newcommand\RRR{\mathbb{R}}
\newcommand\ZZZ{\mathbb{Z}}
\newcommand{\NNN}{\mathbb{N}}

\newcommand{\KR}{\Gamma}
\newcommand{\Stab}{\mathcal{S}}
\newcommand{\Aut}{\mathrm{Aut}}
\newcommand{\Gstab}{G}
\newcommand{\st}{\mathrm{st}}

\newcommand{\clP}{\mathcal{P}}

\newcommand{\Diff}{\mathcal{D}}
\newcommand{\Orb}{\mathcal{O}}
\newcommand{\id}{\mathrm{id}}
\newcommand{\Gclass}{\mathscr{G}}
\newcommand{\E}{\mathcal{E}}
\newcommand{\todo}[1]{}
\renewcommand{\todo}[1]{{{\bf\color{Blue} TODO: {#1}}}}

\begin{document}
\title[Automorphisms of  graphs of Morse functions on $2$-torus]
{Automorphisms of Kronrod-Reeb graphs of  Morse functions on $2$-torus}

\author{Anna Kravchenko}
\address{Taras Shevchenko National University of Kyiv, Volodymyrska St, 60, Kyiv, 01033, Ukraine}
\email{annakravchenko1606@gmail.com}

\author{Bohdan Feshchenko}
\address{Topology laboratory, Department of algebra and topology, Institute of Mathematics of National Academy of Science of Ukraine,
Tereshchenkivska, 3, Kyiv, 01601, Ukraine}
\curraddr{}
\email{fb@imath.kiev.ua}

\subjclass[2010]{}

\begin{abstract}
	This paper is devoted to the study of special subgroups of the automorphism groups of  Kronrod-Reeb graphs of a Morse functions on $2$-torus $T^2$ which arise from the action of diffeomorphisms preserving a given Morse function on $T^2$. In this paper we give a full description of such classes of groups.
\end{abstract}

\maketitle
\section{Introduction}\label{sec:Introduction}
Topological graphs naturally arise from the study of smooth functions on smooth manifolds as the powerful tools which contain  ``combinatorial'' information about a given smooth function and hence the information about a topology of a smooth manifolds. Kronrod-Reeb graphs of  Morse functions on compact manifolds, named after G.~Reeb by R.~Thom and A.~Kronrod by V.~Sharko are the famous example of such graphs.  They were studied by many authors. E.g.
%
the  problem of realization of a graph as a Kronrod-Reeb graph of a Morse (in general smooth) function on a given smooth compact manifold was proposed by V.~Sharko and were studied in papers of V.~ Sharko \cite{Sharko:MFAT:2006}, Y.~Masumoto and O.~Saeki \cite{MasumotoSaeki:2011:Kyushu}, \L{}.~Michalak \cite{Michalak:2018:TopolMethodsNonlinearAnal}, K. Cole-McLaughlin et al \cite{Cole-McLaughlin:2004:DiscreteComputGeom}, M.~Kaluba, W.~Marzantowicz, N.~Silva \cite{KalubaMarzantowicz:2015:TopolMethodsNonlinearAnal} and others. This problem is 
closely related to other important problem of topologically conjugacy for Morse functions on smooth manifolds, see {e.g.} E.~Kulinich \cite{Kulinich:1998:MFAT}, V.~Sharko \cite{Sharko:UMZ:2003}, D.~Lychak and O.~Prishlyak \cite{LychakPrishlyak:2009:MFAT}. E.~Polulyakh \cite{Polulyakh:2015:UMZH, Polulyakh:2016:UMZH}  generalized the notion of Kronrod-Reeb ``graph'' for functions on non-compact surfaces.


Homotopy properties of Morse functions on smooth surfaces were studied by V.~Sharko \cite{Sharko:PrIntMat:1998}, H.~Zieschang, S.~Matveev, E.~Kudryavtseva \cite{Kudryavtseva:MatSb:1999}, K.~Ikegami and O.~Saeki \cite{IkegamiSaeki:JMSJap:2003}, B.~Kalmar \cite{Kalmar:KJM:2005}, S.~Maksymenko.
We give a short overview of the results of S.~Maksymenko  \cite{Maksymenko:AGAG:2006, Maksymenko:MFAT:2010, Maksymenko:ProcIM:ENG:2010, Maksymenko:UMZ:ENG:2012, Maksymenko:DefFuncI:2014} and E.~Kudryavtseva devoted to the study of homotopy properties of orbits and stabilizers of smooth functions on compact surfaces under the action of their diffeomorphism groups. And we will see that  subgroups of the automorphism groups of Kronrod-Reeb graphs plays the essential role in the description of homotopy types of these spaces.

Let $M$ be a smooth compact surface and $X$ be a closed (possible empty) subset of $M$. The diffeomorphisms group $\Diff(M,X)$ acts on the space of smooth functions $C^{\infty}(M)$ by the rule: $C^{\infty}(M)\times\Diff(M,X)\to C^{\infty}(M)$, $(f,h) \mapsto f\circ h$. Under this action  by
$$
\Stab(f,X) = \{h\in \Diff(M,X)\,|\, f\circ h = f\},\qquad\Orb(f,X) = \{f\circ h\,|\, h\in\Diff(M,X)\}
$$
we denote the stabilizer and the orbit of $f\in C^{\infty}(M).$ Endow on $C^{\infty}(M)$ and $\Diff(M,X)$ strong Whitney topologies; for the function $f\in C^{\infty}(M)$ these topologies induce some topologies on $\Stab(f,X)$ and $\Orb(f,X)$.
 By $\Diff_{\id}(M,X)$, $\Stab_{\id}(f,X)$, and $\Orb_f(f,X)$ we denote  connected components of the identity map $\id_M$ of $\Diff(M,X)$, $\Stab(f,X)$, and the component of $\Orb(f,X)$ contains $f$ respectively. If $X=\varnothing$ we put $\Diff(M):=\Diff(M,\varnothing)$, $\Stab(f) := \Stab(f,\varnothing)$, $\Orb(f):= \Orb(f,\varnothing)$, and so on.
 
In this paper we will consider Kronrod-Reeb graphs of  Morse functions on smooth compact surfaces. Since these surfaces may have the boundary we need to specify what  will be meant by a Morse function.
 By a Morse function $f$ on  a surface $M$ we will mean a smooth function $f\in C^{\infty}(M)$ on a surface $M$ such that  $f$ takes constant values on the connected components of the boundary $\partial M$ and each critical point of $f$ is non-degenerate and is contained in $\mathrm{Int}(M).$ A Morse function $f$ on $M$ is called {\it simple} if every critical connected component of  every critical level-set  contains a unique critical point, and  $f$  is called {\it generic} if each level-set of $f$ contains no more than one critical point.
  
S.~Maksymenko,  \cite{Maksymenko:AGAG:2006, Maksymenko:MFAT:2010, Maksymenko:ProcIM:ENG:2010, Maksymenko:UMZ:ENG:2012, Maksymenko:DefFuncI:2014}, showed that if $f$ has at least one saddle point, then $\pi_n\Orb_f(f) = \pi_n M$ for $n\geq 3,$ $\pi_2\Orb_f(f) = 0$, and for $\pi_1\Orb_f(f)$ there is a short exact sequence
 \begin{equation}\label{eq:main}
 \xymatrix{
1\ar[r] & \pi_1\Diff_{\id}(M)\oplus \ZZZ^k\ar[r] & \pi_1\Orb_f(f) \ar[r] & \Gstab(f) \ar[r] & 1 
}
\end{equation}
for some $k\geq 0$ and a finite group $\Gstab(f)$ which is a group of automorphisms of a Kronrod-Reeb graph of $f$ induced by isotopic to $\id_M$ and $f$-preserving diffeomorphisms, i.e., diffeomorphisms from $\Stab'(f) = \Stab(f)\cap \Diff_{\id}(M)$. Moreover if $f$ is generic, then the group $G(f)$ is trivial, and $\mathcal{O}_f(f)$ is homotopy equivalent to $(S^1)^m$ if $M\neq S^2$ and $M\neq\RRR P^2$, to $S^2$ if $M = S^2$ and $f$ have only two critical points, and to $\mathrm{SO}(3)\times (S^1)^m$ otherwise, for some $m\geq 0$ depending on $f$. E.~Kudryavtseva \cite{Kudryavtseva:MathNotes:2012, Kudryavtseva:MatSb:2013} calculated the homotopy types of connected components of the space of Morse functions on compact surfaces, and extend the result about homotopy type of $\mathcal{O}_f(f)$  when $G(f)$ is non-trivial. She showed that  the space $\mathcal{O}_f(f)$ is homotopy equivalent to the following quotient-spaces $(S^1)^m/G(f)$ if $M\neq S^2$, and to $\mathrm{SO}(3) \times (S^1)^m/G(f)$ otherwise, where $G(f)$ freely acts on $(S^1)^m$, and the number $m$ is the rank of the abelian group  $\pi_1\mathcal{D}_{\id}(M)\oplus \ZZZ^k$, see \eqref{eq:main}. Note that $G(f)$ is the holonomy group of the compact manifold $(S^1)^m/G(f)$.

An algebraic structure of $\pi_1\mathcal{O}_f(f)$ for the Morse functions on $2$-torus was described in the series of papers by S. Maksymenko and the second author \cite{MaksymenkoFeshchenko:2014:HomPropCycle, MaksymenkoFeshchenko:2014:HomPropTree, MaksymenkoFeshchenko:2015:HomPropCycleNonTri, Feshchenko:2014:HomPropTreeNonTri}. This is one of non-trivial cases since $\mathcal{D}_{\id}(T^2)$ is not contractible, so the image $\pi_1\mathcal{D}_{\id}(T^2)$ in $\pi_1\mathcal{O}_f(f)$ has non-trivial impact on the algebraic structure of $\pi_1\mathcal{O}_f(f)$, see \eqref{eq:main}.

Since groups $G(f)$ play the essential role in the description of homotopy types of $\mathcal{O}_f(f)$ S.~Maksymenko and A.~Kravchenko \cite{MaksymenkoKravchenko:Graphs:2019} studied classes of such groups and their subsets.
Let $\Gclass(M)$ be the isomorphisms class of groups $\Gstab(f)$ for all Morse functions $f$ on $M$. By $\Gclass^{\mathrm{smp}}(M)$ and $\Gclass^{\mathrm{gen}}(M)$ will be denoted subclasses of $\Gclass(M)$ which corresponds to  isomorphisms classes of groups $\Gstab(f)$ for all simple and generic Morse functions on $M$ respectively. The first author and S.~Maksymenko gave a full algebraic  description of classes $\Gclass(M)$ and $\Gclass^{\mathrm{smp}}(M)$   for all compact oriented surface $M$ which are distinct from $2$-sphere $S^2$ and  $2$-torus $T^2$, and  proved that the class  $\Gclass^{\mathrm{gen}}(M)$ is trivial, i.e., contains only trivial group $\{1\}$ for  all compact oriented surface $M$, see Theorem \ref{thm:classPP2}, or \cite{MaksymenkoKravchenko:Graphs:2019}. We also mention that an algebraic structure of $G(f)$ for Morse function on $S^2$ is partially understood   \cite{MaksymenkoKravchenko:GraphsSp:2018}, and in general this case is more complicated than the case of Morse functions on compact surfaces of genus $\geq 1$. Recently   S.~Maksymenko and A.~Kravchenko \cite{ MaksymenkoKravchenko:GraphsSp:2019} described special subgroups of $G(f)$ for Morse functions on $S^2$.

The aim of this paper is to extend the result of \cite{MaksymenkoKravchenko:Graphs:2019} by  given a {\it full description of the classes $\Gclass(T^2)$ and} $\Gclass^{smp}(T^2)$ for a Morse functions on $2$-torus. 
 
\subsection{Acknowledgments} Authors would like to express their gratitude to Sergiy Maksymenko  for   advices and discussions.

\subsection{Conventions and notations} 
To state our main result we need the notion of wreath products of groups with cyclic groups. So we recall these definitions.
Let $G$ be a group, $m,n\geq 1$ be integers. Consider two effective actions $\alpha:G^n\times \ZZZ_n\to G^n$, and $\beta:G^{nm}\times (\ZZZ_n\times \ZZZ_m)\to G^{nm}$ of $\ZZZ_n$ and $\ZZZ_n\times\ZZZ_{mn}$ on $G^n$ and $G^{nm}$ given by the formulas:
$$
\alpha((g_i)_{i = 0}^{n-1}, a) = (g_{i+a})_{i = 0}^{n-1}, \qquad \beta((g_{i,j})_{i,j=0}^{n-1, m-1}, (b,c)) = (g_{i+b,j+c})_{i,j=0}^{n-1, m-1}
$$
where all indexes are taken modulo $n$ and $n,m$ respectively. With respect to these actions we define semi-direct products $G\wr\ZZZ_n:=G^n\rtimes_{\alpha}\ZZZ_n$ and $G\wr(\ZZZ_n\times\ZZZ_m):=G^{nm}\rtimes_{\beta}(\ZZZ_n\times\ZZZ_m)$   and  call them {\it wreath products} of $G$ with $\ZZZ_n$ and $G$ with $(\ZZZ_n\times\ZZZ_m)$ respectively.
More general definition the reader can find in the book \cite{Meldrum:Longman:1995}.

\subsection{Structure of the paper} Section \ref{sec:Main-result} is devoted to out main result -- Theorem \ref{thm:main}. Sections \ref{sec:Aut-graphs} and \ref{sec:Torus} include some preliminary facts about automorphisms of graphs of Morse functions on surfaces, and ``combinatorial'' structure of such functions on $2$-torus. The proof of Theorem \ref{thm:main} is contained in Section \ref{sec:Proof-Main-Theorem}.

\section{Main result}\label{sec:Main-result}
\subsection{The class $\clP$} First we recall the main result of \cite{MaksymenkoKravchenko:Graphs:2019}. 
For $n\in\NNN$, let $\clP_n$ be a minimal set of isomorphism classes of groups which satisfies the following two conditions:
\begin{enumerate}
	\item the unit group $\{1\}$ belongs to $\clP$,
	\item  groups $A\times B$ and $A\wr \ZZZ_n$ belong to $\clP$ whenever $A, B\in\clP$, and $n\in \NNN$.
\end{enumerate}
Let also $\clP$ be a minimal set of isomorphism classes of groups which contains $\clP_n$ as a set for each $n\in\NNN.$
The following theorem gives the description of the classes $\Gclass(M)$, $\Gclass^{smp}(M)$ and $\Gclass^{gen}(M)$ for all compact surfaces $M\neq S^2, T^2$.
\begin{theorem}[Theorem 1.4 \cite{MaksymenkoKravchenko:Graphs:2019}]\label{thm:classPP2}
	For each compact oriented surface $M\neq S^2, T^2$, $\Gclass(M) = \clP$, and  $\Gclass^{\mathrm{smp}}(M) = \clP_2$. For all compact oriented surface $\Gclass^{gen}(M) = \{1\}$.
\end{theorem}

\subsection{Classes of groups $\E_1$, $\E_2$ and $\E_2$}
To describe classes $\Gclass(T^2)$ and $\Gclass^{smp}(T^2)$ we introduce three  classes $\E_0$, $\E_1$ and $\mathcal{E}_2$ in the following way. 
Let $\E_i$ be a  minimal set of isomorphism classes of groups such that
\begin{itemize}
	\item $\E_0$ contains the group $A_0\wr(\ZZZ_n\times\ZZZ_{mn})$, $n,m\geq 1$ for each $A_0\in \clP$,
	\item $\E_1$ contains the group $A_1\wr\ZZZ_n$, $n\geq 1$ for each $A_1\in \clP$,
	\item $\E_2$ contains the group $A_2\wr\ZZZ_n$, $n\geq 1$ for each $A_2\in \clP_2$.
\end{itemize}
\begin{remark}{\rm
Note that the group $(A\wr\ZZZ_2) \times (B\wr\ZZZ_3)$ where $A,B\in \clP$  belongs to $\E_1$ since it is isomorphic to the group $((A\wr\ZZZ_2) \times (B\wr\ZZZ_3)) \wr\ZZZ_1$.
So classes  $\clP$ and $\E_1$ in general coincide, but we will consider the class $\E_1$ due to the convenience of a unique presentation of a group $G\in \E_1$ in the form $G = H\wr\ZZZ_n$ for some $n\geq 1$ and $H\in\clP$.}
\end{remark}

It is well known that  the Kronrod-Reeb graph of Morse functions on $2$-torus is either a tree or contains a unique circuit, see Lemma 3.1 in \cite{Feshchenko:DefFuncI:2019}. The class of groups $G(f)$ for Morse functions  $f$ on $T^2$ whose graph are trees we will denote by $\Gclass_0(T^2)$, otherwise, in the case of circuits, by $\Gclass_1(T^2)$\footnote[1]{Here indexes $0$ and $1$ correspond to the rank of $H_1(\Gamma_f,\ZZZ)$}. The following theorem is our main result.
\begin{theorem}\label{thm:main}
	The following classes coincide: 
	$$\Gclass_0(T^2) = \E_0,\qquad \Gclass_1(T^2) = \E_1,\qquad \Gclass^{gen}(T^2) = \E_2.$$
\end{theorem}
We prove Theorem \ref{thm:main} in Section \ref{sec:Proof-Main-Theorem}. The proof is divided into three separate cases: for $\Gclass_0(T^2)$, $\Gclass_1(T^2)$ and $\Gclass^{gen}(T^2)$.

\section{Automorphisms of graphs of functions on surfaces}\label{sec:Aut-graphs}
In this section we want to show a precise way how the group $G(f)$ arises from the action of $\mathcal{S}'(f) = \mathcal{S}(f)\cap \mathcal{D}_{\id}(M)$ on $M$.
Let $f:M\to \RRR$ be a Morse function on smooth oriented surface $M$, and $c$ be a real number.  A connected component $C$ of the level-set $f^{-1}(c)$ is called {\it critical}, if $C$ contains  at most one critical point of $f$, otherwise $C$ is called {\it regular}. Let $\Delta$ be a partition of $M$ into connected components of level-sets of $f$. It is well known that the quotient-space $\KR_f = M/\Delta$ has a structure of an $1$-dimensional  CW complex called a  Kronrod-Reeb graph of $f$. For simplicity we will call it a graph of $f$. 
Let also $p_f:M\to \Gamma_f$ be a  projection map. Then $f$ can be presented as the composition: $f = \widehat{f} \circ p_f: M \stackrel{~~p_f~~}{\longrightarrow} \KR_f \stackrel{~~\widehat{f}~~}{\longrightarrow}\RRR$.
Denote by $\Aut(\KR_f)$ the group of homeomorphisms of the graph $\KR_f$. Note that each $h\in \Stab'(f)$ preserves level-sets of $f$. Hence, $h\in \Stab'(f)$ induces the homeomorphism $\rho(h)$ of $\KR_f$ such that the following diagram 
$$
\xymatrix{
M\ar[r]^{p_f} \ar[d]_h & \KR_f \ar[r]^{\widehat{f}} \ar[d]_{\rho(h)} &\RRR \ar@{=}[d]\\
M\ar[r]^{p_f} & \KR_f \ar[r]^{\widehat{f}} & \RRR
}
$$
commutes, and the correspondence $h\mapsto \rho(h)$ is a homeomorphism $\rho:\Stab'(f)\to \Aut(\KR_f)$.
The image $\rho(\Stab'(f))$ is a finite group in $\Aut(\KR_f)$; we will denote it by $\Gstab(f)$.

\section{Combinatorial generalities on Morse functions on $2$-torus and  their graphs}\label{sec:Torus}
We are interested in the ``combinatorial'' structure of Morse functions on $T^2$, so we will recollect some useful for us results on the structure of such functions. The following lemma holds.
\begin{lemma}[Lemma 3.1 \cite{Feshchenko:DefFuncI:2019}]
Let $f$ be a Morse function on $T^2$ and $\Gamma_f$ be its graph.  Then $\Gamma_f$ is either a tree or contains a unique circuit.
\end{lemma}
\noindent We describe these two cases separately.

\subsection{$\Gamma_f$ contains a circuit}
\label{subsec:circ}
 Let $\Theta$ be a circuit in $\Gamma_f$.
Let $C_0\subset T^2$ be a regular connected component of some level set $f^{-1}(c)$, $c\in \RRR$, and $z$ be a point in $\KR_f$ corresponding to $C_0$. Obviously, $z$ belongs to the cycle $\Theta$ in $\KR_f$, iff $C_0$ does not separate $T^2$. Note that the level-set $f^{-1}(c)$ consists of a finite number of connected components, and is invariant under the action of any $h\in\Stab'(f)$. Let $\mathcal{C}$ be the set $\{h(C_0)\,|\, h\in\Stab'(f) \}$ of all images of $C_0$ under the action of elements from $\Stab'(f)$. Then the set $\mathcal{C}$ consists of a finite number of components $\{C_0,C_1,\ldots,C_{n-1}\}$ of the set $f^{-1}(c)$ for some $n\geq 1$. Curves from $\mathcal{C}$ are pairwise disjoint, and since $C_0$ does not separate $T^2$, it follows that each $C_i$ also does not separate $T^2$. Then $C_i$ and $C_{i+1}$ bounds a cylinder $Q_i$ such that the interior of $Q_i$ does not intersects with $\mathcal{C}.$
Note that the group $\ZZZ_n$ freely acts on the set of cylinders $\{Q_i\}$ by cyclic permutations. More about combinatorial description of Morse functions on $2$-torus whose graphs contain circuits the reader can find in \cite{MaksymenkoFeshchenko:2014:HomPropCycle, MaksymenkoFeshchenko:2015:HomPropCycleNonTri}.

\subsection{$\Gamma_f$ is a tree}
\label{subsec:tree}
 Let $f$ be a Morse function on $T^2$ with $\Gamma_f$ which is a tree. Then by \cite[Theorem~2.5]{Feshchenko:2014:HomPropTreeNonTri} there exists a unique vertex $v\in \Gamma_f$ such that each connected component of $T^2 - p_f(v)$ is an open $2$-disk.
Such vertex $v$ of $\Gamma_f$ and the connected component  $V = p_f^{-1}(v)$ of $f$ which corresponds to $v$ will be called {\it special}. Note that the  topological structure of the atom of $V$, i.e., a regular neighborhood of $V$ which consists of connected components of level-set of $f$ and does not contain other critical points of $f$, is well understood \cite{Feshchenko:2016:ActTree}.

Next we describe a special subgroup of $G(f)$ which plays an important role in an algebraic description of $G(f)$. The group $\Gstab(f)$ acts on the graph $\KR_f$, and we let 
$\Gstab_v$  be a stabilizer of $v$ with respect to this $\Gstab$-action on $\KR_f$. The set $\Gstab_v^{loc} = \{ g|_{\st(v)}\,|\, g\in \Gstab_v \}$ which contains restrictions of elements of $\Gstab_v$ onto the star $\st(v)$ is a subgroup of $\Aut(\st(v))$,  we will call it a {\it local stabilizer of $v$}. It is well known that the group $G_v^{loc}$ is isomorphic to the product $\ZZZ_n\times \ZZZ_{mn}$ for some $n,m\geq 1$, see \cite[Theorem~2.5]{Feshchenko:2014:HomPropTreeNonTri}. More informations the reader can find in \cite{Feshchenko:2016:ActTree, MaksymenkoFeshchenko:2014:HomPropTree}.

\subsection{Algebraic structure of $G(f)$}
It turns out that the combinatorial information from the subsections \ref{subsec:circ} and \ref{subsec:tree} is enough to describe   an algebraic structure of the groups $G(f)$ for Morse functions on $T^2$.
\begin{lemma}[Theorem 3.2 \cite{Feshchenko:DefFuncI:2019}]
	
	\label{prop:GraphT} 
	Let $f$ be a Morse function on $T^2$, and $\Gamma_f$ be its graph. 
	\begin{enumerate}
				\item 	If  $\KR_f$ contains a unique circle,  then the group $\ZZZ_n$ freely acts on the set of cylinders $\{Q_i\}_{i = 0}^{n-1}$ and there is an isomorphism  $\Gstab(f) = \Gstab(f|_{Q_0})\wr \ZZZ_n$, where $n$ is a cyclic index of $f$,and $Q_0$ is a cylinder bounded by parallel curves $C_0$ and $C_1$.
		\item 
		If $\KR_f$ is a tree,  then there exists a special vertex $v$ in $\Gamma_f$. Let also $V = p_f^{-1}(v)$ be a special component of critical level set of $f$ which corresponds to $v$ and $N$ be an atom of $V$.
		Moreover the group $G_v^{loc}$  freely acts by diffeomorphisms from $\mathcal{S}'(f)$ on $T^2$, 
		this action induces a free action of $G_v^{loc}$ on connected components of $\overline{T^2-N}$,
		 and  so
		 there exists a set of  $2$-disks $\{ D_0,D_1,\ldots, D_r\}$ such that $\Gstab(f) = \prod_{i = 0}^{r} \Gstab(f|_{D_i})\wr G_v^{loc}$, where $r$ is a number of orbits of free $\Gstab$-action on $T^2$.
	\end{enumerate}
\end{lemma}

\subsection{Morse equality} 
We would like to recall the relationship between  Euler characteristic of the surface and Morse function  defined on it. This  connection is given by the outstanding Morse equality -- it is an important ingredient needed for the proof of Theorem \ref{thm:main}. 
\begin{theorem}[Morse equality]\label{thm:Morse}
	Let $f$ be a Morse function on smooth compact and oriented surface $M$ without boundary and let $c_i(f)$ be a number of critical points of $f$ of index $i$, $i = 0,1,2$. Then the following equality holds
	$$
	\chi(M) = c_0(f) - c_1(f) + c_2(f).
	$$
\end{theorem}
It is well known that $\chi(T^2) = 0$,  so Morse equality for Morse functions on $T^2$ has the form: $c_0(f)  + c_2(f) = c_1(f)$.

\section{Proof of Theorem \ref{thm:main}}\label{sec:Proof-Main-Theorem}
Obviously that the inclusion $\Gclass_i(T^2)\hookrightarrow \E_i$ for $i = 0,1$ directly follows from Lemma  \ref{prop:GraphT} because of the structure of the class $\E_i$. So to prove Theorem \ref{thm:main} we have to establish the reverse inclusion $\E_i\hookrightarrow \Gclass_i(T^2)$ for $i = 0,1$. In other words we need to prove that for each $A\in \E_i$ there exists a Morse function $f$ on $T^2$ such that $A\cong G(f)$. These will be done bellow in Subsections \ref{subsec:treecirc} and \ref{subsec:cyclecirc}.
The case of simple Morse functions will be considered in Subsection \ref{subsec:cimple}.
\subsection{Case 1: Functions, whose graphs have circuits}\label{subsec:treecirc} Let $A$ be a group from $\E_1$. So the group $A$ has the form $B\wr \ZZZ_n$ for some $B\in \clP$ and some $n\geq 1.$ We divide the proof of the inclusion $\E_1\hookrightarrow \Gclass_1(T^2)$ into two steps.
First we have to define a Morse function function $f_0:T^2\to \RRR$ such that $\Gamma_{f_0}$ contains a circuit and $G(f_0)\cong \ZZZ_n$ for $n$ as above. Then for the given group $B$ we change $f_0$ on the neighborhood of maxima of $f_0$  to obtain another Morse function $f:T^2\to \RRR$ such that $G(f)\cong B\wr \ZZZ_n$.

Such Morse function $f_0$  one can define by the following procedure. Let $Q_i\cong S^1\times [0,1]$, $i = 0,\ldots, n-1$ be a cylinder and $g_i:Q_i\to \RRR$ be a Morse function such that $g_i$ has one maximum $l_i$, one minimum $p_i$, two saddles $s_{i1}, s_{i2}$ and  
$$
g(\partial Q_i) = 0,\qquad g(l_i) = 1,\qquad g(p_i) = -1,\qquad g(s_{ij}) = \pm 1/2, \;j = 1,2.
$$
Consider the same orientation on each  $Q_i$; it canonically induces the orientations of the connected components of the boundary $\partial Q_i = \partial Q_i^{0}\cup \partial Q_i^{1},$ where $\partial Q_i^{j} = S^1\times \{j\}$, $j = 0,1$.

If $n = 1$ attach the boundaries of $Q_0$ by identity diffeomorphism. The resulting surface is a $2$-torus. Since the values of $g$ on $\partial Q_0^0$ and $\partial Q_0^1$ coincides, it follows that $g_0$ induces a unique Morse function $f_0$ on $T^2$.
If $n\geq 1$ attach all $Q_i$ together in cyclic order by identity diffeomorphism of $\partial Q_i^{1}\to \partial Q_{i+1}^0$ where the index $i$ takes modulo $n$. The resulting surface is obviously a $2$-torus, and since the values $g_i$ on connected components of the boundary of $Q_i$ coincides, it follows that $g_0$ induces a unique smooth function $f_0$ on $T^2$ such that $\Gamma_{f_0}$ contains a circuit and $G(f_0)\cong \ZZZ_n$. Points $l_i$, $p_i$ and $s_{ij}$ are the corresponding maxima, minima and saddle points of $f_0$, $i = 0,\ldots n-1$, $j = 1,2$.

Let $D_i$ be a regular neighborhood of maximum $l_i$ which does not contain other critical points. For the given group $B$ by Theorem \ref{thm:classPP2}  there exists a smooth function $f_i$ on $D_i$ such that $G(f_i) \cong B$, where $B\in \clP.$ Next we change the function $f_0$ on $D_i$ by replacing $f_0|_{D_i}$ to $f_i$ on $D_i$; the resulting function we denote by $f$. By Lemma \ref{prop:GraphT}, $f$ is such that $G(f) \cong B\wr \ZZZ_n$. So we proved that the inclusion $\E_1\hookrightarrow \Gclass_1(T^2)$ holds.

\subsection{Case 2: Functions, whose graphs are trees}\label{subsec:cyclecirc}  Let $A$ be a group from $\E_0$. We need to show that there exists a Morse functions $f$ on $T^2$ such that $\Gamma_f$ is a tree and $G(f) \cong A$. From the definition of the class $\E_0$ there exist $n,m\geq 1$ and $B\in \clP$ such that $A = B\wr (\ZZZ_n\times\ZZZ_{mn}).$ As in the Case 1 we divide our proof into two steps. First for a given $n,m\geq 1$ we define a Morse function $f_0:T^2\to \RRR$ such that $\Gamma_{f_0}$ is a tree and $G(f_0) = \ZZZ_n\times \ZZZ_{mn}$, and finally for the given group $B$ we change the function $f_0$ to obtain the Morse function $f:T^2\to \RRR$ such that $G(f)\cong B\wr (\ZZZ_n\times\ZZZ_{mn})$.

To realize the first step we need some preliminaries.
Let $\gamma:\RRR^2\times \ZZZ^2\to \RRR^2$ be a free action of $\ZZZ^2$ given by the formula
$\gamma((x,y), (1,1)) = (x+2, y+2)$. For  $n,m\geq 1$ as above this action induce a free action $\delta$ of the subgroup $n\ZZZ\times mn\ZZZ$ of $\ZZZ^2$ by the formula
$\delta((x,y), (1,1)) = (x + 2n, y+ 2mn).$ Note that the rectangle $P = [0,2n]\times [0,2mn]$ is a fundamental domain for the action $\delta,$ and   the quotient space of $\RRR^2/\delta$ is a $2$-torus. Let  $D_{kl}$ be a square $[k,k+1]\times [l,l+1]\subset \RRR^2$, $k,l\in\ZZZ$, and $q:\RRR^2\to \RRR^2/\delta$ be the projection map, and $\mathsf{B}$ be the set of images of $D_{kl}$ with respect to $q.$
 The action $\delta$ induces the free action $\sigma$ on $\ZZZ_n\times\ZZZ_{mn}$ on the set $\mathsf{B}$. The number of orbits of this action is equal to 
$$
\frac{\#(\text{squares } D_{kl} \text{ belonging to } P)}{\#(\ZZZ_n\times \ZZZ_{mn})} = \frac{4n^2m}{n^2m} = 4.
$$
So all disks $D_{kl} = [k,k+1]\times [l,l+1]$ in $P$ can be enumerated as $D_{rij}$ by three indexes $r = 1,2,3,4$, $i = 0, 1, \ldots, n-1$, and $j = 0,1,\ldots, mn-1$. We denote by $\mathsf{D}_k$ the set $\{D_{kij}\,|\,i = 0,1,\ldots,n-1,\, j = 0,1,\ldots, mn-1\}$ of all squares which belong to the same orbit, $r = 1,2,3,4$, and let $m_{rij}$ be the internal point of $D_{rij}.$

 Then there exists a double-periodic Morse  function $g_0:\RRR^2\to \RRR$ such that $g_0\circ \delta = g_0$, and   on $P$ it  satisfies:
\begin{itemize}
	\item $g_0(\{k,l\}) = 0$ is a saddle point, $k = 0,1,\ldots, 2n-1$, $l = 0,1,\ldots, 2mn-1$,
	\item $g_0$ has a unique maximum at $m_{r00}$ for $r = 1,2$, and a unique minimum for $r = 3,4$ and such that 
	$$
	g_0(m_{100}) = 1, \qquad g_0(m_{200}) = 2, \qquad g_0(m_{300}) = -1, \qquad g_0(m_{400}) = -2.
	$$
\end{itemize}
It is easy to see that $g_0$ induces a unique Morse function on $T^2 = \RRR^2/\delta$, which we  denote by $f_0$. By construction, $\Gamma_{f_0}$ is a tree, and $G(f_0) = \ZZZ_n\times\ZZZ_{mn}$.  

It remains to modify $f_0$ using the group $B$ as above.
 Since the action $\sigma$  on $\mathsf{B}$ has $4$ orbits, choose one of them, say when $r = 1$. 
  By Theorem \ref{thm:classPP2} there exists a Morse function $f_{1ij}$ on $D_{1ij}$ such that $G(f_{1ij}) \cong B$ for all $i = 0,\ldots, n-1$ and $j = 0,\ldots, mn-1$. We change the function $f_0$ on $D_{1ij}$ such that $f_0|_{D_{1ij}} = f_{1ij}$, and the resulting functions we will denote by $f$.  Note that $G(f|_{D_{rij}}) = 1$ for $r = 2,3,4$ and $i = 0,1\ldots, n-1$, $j = 0,1,\ldots, nm-1$ by definition of $f_0$.
  Then by Lemma \ref{prop:GraphT}
  \begin{align*}
  G(f)&\cong \prod_{r = 1}^4 G(f|_{D_{r00}})\wr (\ZZZ_n\times\ZZZ_{mn})\\
  &=G(f|_{D_{100}})\wr (\ZZZ_n\times\ZZZ_{mn})\\
  &\cong B\wr (\ZZZ_n\times\ZZZ_{mn}).
  \end{align*}
So we proved that the inclusion $\E_0\hookrightarrow \Gclass_0(T^2)$ holds.

\subsection{Case 3: Simple Morse functions}\label{subsec:cimple}
It is easy to see that there exist simple Morse functions on $T^2$ in the case when its graph contains a unique circuit. Indeed, a Morse function $f$  on $T^2$ such that $\Gamma_f$ contains a unique circuit is simple if the restriction $f|_{Q_i}$ is simple for each $i = 0,1\ldots, n$. 
The following lemma shows that this is the only case.
\begin{lemma}\label{lm:simple}
	Let $f$ be a simple Morse function on $T^2$. Then $\Gamma_f$ is not a tree.
\end{lemma}
\begin{proof}
	Assume that $f$ is simple and $\Gamma_f$ is a tree. Then there exists a unique special vertex $v$ of the tree $\Gamma_f$. Let $V = p_f^{-1}(v)$ be a special component of critical level-set of $f$ which corresponds to $v$. Note that the special component $V$ contains ``many'' saddles of $f$. Let also $N$ be a regular neighborhood of $V$ which consists of level-sets of $f$ and does not contain other critical points. Since $v$ is a special vertex, it follows that $\overline{T^2-N}$ is a disjoint union of  $2$-disks, say $\{D_i\}_{i = 1}^n$ for some $n\in\NNN.$  Note that the restriction $f|_{D_i}$ for all $i = 0,1,\ldots n$ is also a simple Morse function since $f$ is simple. 
	
	Next we change the function $f$ in the following way: replace $f|_{D_i}$ by the function on $2$-disk $D_i$ which has the only one critical point, $i = 1,\ldots, n$; the resulting function we denote by $g$. After these changes the function $g$ satisfies the following conditions:
	\begin{enumerate}
		\item $\Gamma_g$ is also a tree, since all changes of $f$ were on the connected components of the complement of $N$ being  $2$-disks,
		\item $V$ is also a critical level-set of $g$, but all saddles of $g$ belong to $V$,
		\item $g$ is simple since $f$ was simple.
	\end{enumerate}
	Since $g$ is simple (by assumption $f$ and by construction), it follows that the number of saddles belonging to $V$ must be equal to $1$, and so $c_1(g) = 1$. From Morse equality (Theorem \ref{thm:Morse})  we have $c_1(g) = c_0(g)+c_2(g)$, and so $1 = c_0(g) + c_2(g)$.  This is an inconsistency. Hence if $f$ is simple, then $\Gamma_f$ is not a tree.
\end{proof}

Now we need to show that the classes $\Gclass^{smp}(T^2)$ and $\E_2$ coincide. 
First we show the inclusion $\Gclass^{smp}(T^2) \hookrightarrow \E_2$ holds.
Let $f$ be a simple Morse function on $T^2$. By Lemma \ref{lm:simple} the graph $\Gamma_f$ has a unique circle. The restriction $f|_{Q_0}$ is simple Morse function on cylinder $Q_0$ so  $G(f|_{Q_0})$ belongs to the class $\mathcal{P}_2$. By (1) Lemma \ref{prop:GraphT} the group $G(f)$ is isomorphic to $G(f|_{Q_0})\wr \ZZZ_n$ for some $n$ which depends on the function $f$, and so $G(f)$ belongs to $\E_2$ by definition of the class $\E_2$.

The reverse inclusion $\E_2\hookrightarrow \Gclass^{smp}(T^2)$ follows from the procedure defined in Case 1 with $B\in \mathcal{P}_2$ as sub-case. Theorem is proved.


\end{document}